\documentclass{article}
\usepackage[subpreambles=true]{standalone}
\usepackage[letterpaper, margin=1.2in]{geometry}
\usepackage[utf8]{inputenc}
\usepackage{import}
\usepackage{amsmath}
\usepackage{amssymb}  
\usepackage[hidelinks]{hyperref}
\usepackage{graphicx}
\usepackage{listings}
\usepackage{amsthm}
\usepackage{amsfonts}
\usepackage{color}
\usepackage{float}
\usepackage{xspace}
\xspaceaddexceptions{]\}}
\usepackage{caption}
\usepackage{subcaption}
\usepackage[font={small}]{caption}
\usepackage{comment}

\usepackage{chngcntr}
\counterwithin{figure}{section}
\counterwithin{table}{section}

\usepackage{esvect}
\usepackage{todonotes}

\usepackage{xcolor}
\usepackage{lipsum}
\usepackage{amsfonts,amsmath,amssymb,amsthm}
\usepackage[shortlabels]{enumitem}
\usepackage{graphicx}
\usepackage{float}
\usepackage{url}

\usepackage[pagewise]{lineno}
\usepackage{todonotes}
\usepackage{setspace}
\usepackage{lineno}

\allowdisplaybreaks

\theoremstyle{plain}
\newtheorem{theorem}{Theorem}[section]
\newtheorem{lemma}[theorem]{Lemma}

\newtheorem{corollary}[theorem]{Corollary}

\newtheorem{remark}[theorem]{Remark}
\newtheorem{definition}[theorem]{Definition}
\newtheorem{property}[theorem]{Property}

\newcommand{\floor}[1]{\lfloor #1 \rfloor}
\newcommand{\ceil}[1]{\lceil #1 \rceil}
\title{On the Sum of Ricci-Curvatures for Weighted Graphs }
\author{Shuliang Bai
\thanks{Harvard University, Cambridge, MA 02138,
({\tt sbai525@cmsa.fas.harvard.edu})
}
\and
An Huang
\thanks{Brandeis University, Waltham, MA 02453,
({\tt anhuang@brandeis.edu})
}
\and
Linyuan Lu
\thanks{University of South Carolina, Columbia, SC 29208,
({\tt lu@math.sc.edu}). 
} 
\and
Shing-Tung Yau
\thanks{Harvard University, Cambridge, MA 02138,
({\tt yau@math.harvard.edu})}
}
\begin{document}
\maketitle

\begin{abstract}
In this paper, we generalize Lin-Lu-Yau's Ricci curvature to weighted graphs and give a simple limit-free definition. 
    We prove two extremal results on the sum of Ricci curvatures for weighted graph.

 A weighted graph $G=(V,E,d)$ is an undirected graph $G=(V,E)$ associated with a distance function $d\colon E\to [0, \infty)$.
    By redefining the weights if possible, without loss of generality, we assume that the shortest weighted distance between $u$ and $v$ is exactly $d(u,v)$ for any edge $uv$.
    Now consider a random walk whose transitive probability 
    from an vertex $u$ to its neighbor $v$ (a jump move along the edge $uv$) is proportional to $w_{uv}:=F(d(u,v))/d(u,v)$ for some given function $F(\bullet)$.
    We first generalize Lin-Lu-Yau's Ricci curvature definition to this weighted graph and give a simple limit-free representation of $\kappa(x, y)$ using a so called $\ast$-coupling functions.
    The total curvature $K(G)$ is defined to be the sum of Ricci curvatures over all edges of $G$.  We proved the following theorems: if $F(\bullet)$ is a decreasing function, then $K(G)\geq 2|V| -2|E|$; if $F(\bullet)$ is an increasing function, then $K(G)\leq 2|V| -2|E|$. Both equalities hold if and only if $d$ is a constant function plus the girth is at least $6$. 
    
    In particular, these imply a Gauss-Bonnet theorem for (unweighted) graphs with girth at least $6$, where the graph Ricci curvature is defined geometrically in terms of optimal transport.

\end{abstract}

\section{Introduction}
Ricci curvature is a fundamental concept from Riemannian Geometry\cite{Jost}
that  has been extended to a discrete setting.  There  are different definitions of  Ricci curvature defined on graphs,  see references \cite{Fchung, LYau, Ollivier1}. Among the
various curvature notions the Ollivier Ricci curvature,   is defined on arbitrary 
metric spaces equipped with a Markov chain,  and has extended some of results
for positively curved manifolds such as the Bonnet-Myers theorem bounding the diameter of the space via curvature, the Lichnerowicz theorem for the spectral gap of the Laplacian, a control on mixing properties of Brownian motion and the Levy-Gromov theorem for isometric
inequalities and concentration of measures\cite{Ollivier2}.  
In the special setting of graphs, the Ollivier Ricci curvature  is based on optimal transport
of  probability measures associated to a lazy random walk \cite{Ollivier1, Ollivier2}.  Analogously to a Riemannian manifold the Ricci curvature defined on  Riemannian manifold measures the local amount of non-flatness of the manifold, while the Ollivier Ricci curvature measures the distance (via the Wasserstein transportation distance) between two small balls centered at two given nodes.  By this notion, positive curvature implies that the neighbors of the two centers are close or overlapping, negative curvature implies that the neighbors of two centers are further apart, and zero curvature or near-zero curvature implies that the neighbors are locally embeddable in a flat surface. 
The Ollivier Ricci curvature provides a curvature of any two nodes and it depends on an idleness
parameter of the random walk. In 2011, Lin, Lu, and Yau \cite{LLY} modified this notion to a limit version so that it does not depend on the idleness parameter, which is more suitable for graphs, such as computing the curvature on random graphs or Cartesian product of graphs. 
Later on, many  properties and  consequences of the Ollivier Ricci curvature and the modified version have been done, see \cite{BCLMP, BJL, BBS,  JL, Smith}, etc. More recently, these curvatures has been applied in various research areas such as network
analysis\cite{NLLG, JJBP}, quantum computation, dynamic Networks\cite{p-adic}, etc.

When it comes to the applications of Ricci curvature, the weighted graph models are more useful than the unweighted graphs models, as in the real-world networks, not all relation have the same capacity. For this, the Ricci curvatures of graphs have been generalized to weighted graphs according to different needs, see \cite{FlorentinRadoslaw, p-adic}. 
In this paper, we study a more general definition of Ricci curvature defined on weighted graphs.  For any weighted graph, there are two symmetric positive valued functions $d, w$ defined on edges, the $d(i, j)$ represent the distance between $i, j\in V$ and $w_{ij}$ represent weight distribution on edge $(i,j)$ which is used to define the probability distribution functions. 
For any vertex $x\in V$ and any value $\alpha \in [0, 1]$,  the  probability distribution $\mu_x^{\alpha}$ assigns amount  $\alpha$ at vertex $x$ and amount $\frac{(1-\alpha)w_{xi}}{\sum_{y\sim x}w_{xy}}$ to all its neighbors $i$. 

Then $\alpha$-Ricci-curvature $\kappa_{\alpha}$ of edge $(x, y)$ is defined to be
$$\kappa_{\alpha}(x, y)=1-\frac{W(\mu_x^{\alpha}, \mu_y^{\alpha})}{d(x, y)},  $$
where $W(\mu_x^{\alpha}, \mu_y^{\alpha})$ is the Wasserstein transportation distance transporting  $\mu_x^{\alpha}$ to $\mu_y^{\alpha}$. 
By Lin-Lu-Yau's definition, the Ricci curvature $\kappa(x, y)$ is defined as
$$\kappa(x, y)=\lim\limits_{\alpha\to 1} \frac{k_{\alpha}(x, y)}{1-\alpha}. $$

Given a weighted graph with fixed $d$, the curvature $\kappa(x, y)$ is a multi-variate function with variables $w_{ij}$. 
We consider $K(G)=\sum_{x\sim y} \kappa(x, y)$: the sum of Ricci-curvatures over all edges of graph $G$. 
%For any fixed weighted graph $G$ with $m$ edges, then $K(G)$ is a function of $m$ variables $w_{e_1}, \ldots, w_{e_m}$, write $\mathbf{w} = (w_{e_1}, \ldots, w_{e_m})$, then we want to study the value $K(G)(\mathbf{w})$. 
It is interesting to know the extremal value of $K(G)$ and the conditions for weighted graph $G$ to achieve these values. As the weight distribution function $w$ varies, the behavior of the extremal value of $K(G)$ changes. In this paper we study the maximal value and the minimal value of $K(G)$ in two different cases, and prove a version of the Gauss-Bonnet theorem for graphs with girth at least 6: Corollary \ref{coro:Gauss}.

The paper is organized as follows, in Section \ref{sec:notations}, we set up the notations of generalized Ricci curvature defined on weighted graphs and compare our definition with the existing ones, we also give a more simple expression of the generalized Ricci curvature using $\ast$-coupling function; in Section \ref{sec:sum}, we state and prove the results about the minimal and maximal of total curvature $K(G)$.

\section{Notations}\label{sec:notations}
In this section, we generalize the definition of Ricci-curvature of graphs to the weighted graphs. A weighted graph $G=(V, E, d)$ is a connected simple graph on  vertex set $V$ and edge set $E$ where set $E$ is associated by the distance function (or edge length function) $d: E\to R^{+}$ which assigns a positive value to each edge $e\in E$. For any two adjacent vertices $x, y$, we represent the length of edge $e=(x, y)$ as  $d(x, y)$ or $d(e)$. 
We call $G$ as a  {\it combinatorial graph} if the distance function $d$ is uniform on all edges, that is $d(e)=1$ by a scaling for all edges $e\in E$. The length of a path is the sum of  edge lengths on the path, for any two non-adjacent vertices $x, y$,  the distance $d(x, y)$ is the length of a minimal weight path among all paths that connect $x$ and $y$.  
For any vertex $x, y\in V$, notation $x\sim y$ represents that two vertices $x$ and $y$ are adjacent, $\Gamma(x)$ represents the set of vertices that are adjacent to $x$ and  $N(x)=\Gamma(x)\cup \{x\}$. 
In this paper, we study the undirected weighted graph, that is,  $d(x,y)=d(y, x)$ for all $(x, y)\in E$.
 The girth of a weighted graph, denoted $girth(G)$ is the size of the smallest
cycle  contained in the combinatorial graph. If the graph does not contain any cycles (i.e. it's an acyclic graph), its girth is defined to be infinity.

We introduce another positive symmetric function defined on the edges of graph $G$, which is used to define the probability distribution function $\mu_x$.
We call it as {\it weight distribution} function $w: E\to R^{+}$.  For better distinction,  we write the value of $w$ on edge $(x, y)$ as $w_{xy}$. Let $D_x=\sum_{y\sim x}w_{xy}$, then
for any vertex $x\in V$ and any value $\alpha \in [0, 1]$,  the  probability distribution $\mu_x^{\alpha}$ is defined as:
\begin{align}\label{equ:defofdistri}
\mu_x^{\alpha}(z)=
\begin{cases}
 \alpha,  & \text{if $z=x$}, \\
 (1-\alpha)\frac{w_{xz}}{D_x},  & \text{if $z\sim x$},\\
 0,& \text{otherwise}.
\end{cases}
\end{align}

\begin{definition}\label{probabilitydistribution}
 Let $G=(V, E, d)$ be a weighted graph.   A probability distribution over the vertex set $V$ is a mapping $\mu: V\to [0,1]$ satisfying $\sum_{x \in V} \mu (x)=1$. Suppose that two probability distributions $\mu_1$ and $\mu_2$ have finite support. A coupling between $\mu_1$ and $\mu_2$ is a mapping $A: V\times V\to [0, 1]$ with finite support so that
$$\sum\limits_{y \in V} A(x, y)=\mu_1(x) \ \text{and}\  \sum\limits_{x \in V} A(x, y)=\mu_2(y). $$

The transportation distance between two probability distributions $\mu_1$ and $\mu_2$ is defined as follows:
\begin{equation}\label{equ:transcoupling}
W(\mu_1, \mu_2)=\inf\limits_{A} \sum\limits_{x, y\in V} A(x, y)d(x, y),\end{equation}

where the infimum is taken over all coupling $A$ between $\mu_1$ and $\mu_2$.
\end{definition}

A coupling function provides a lower bound for the transportation distance, the following definition can provide an upper bound for the transportation distance.
%Another way to compute the transportation distance is using the {\it Lip(1)schitz function}:
\begin{definition}\label{Lipschitz}
Let $G=(V, E, d)$ be a locally finite weighted graph. Let $f: V \to \mathbb{R}$. We say $f$ is
Lip(1)schitz if
$$f(x)-f(y)\leq d(x, y),$$
for each $x, y\in V$.
\end{definition}
By the duality theorem of a linear optimization problem, the transportation distance can also be written as follows:
\begin{equation}\label{equ:transdistance}
W(\mu_1, \mu_2)=\sup\limits_{f} \sum\limits_{x\in V} f(x)[\mu_1(x)-\mu_2(x)],
\end{equation}
%$$$W(\mu_1, \mu_2)=\sup\limits_{f} \sum\limits_{x\in V} f(x)[\mu_1(x)-\mu_2(x)],$$
where the supremum is taken over all Lip(1)schitz functions $f$.
We will call a $A\in V\times V$ satisfying the above infimum in equation \ref{equ:transcoupling} an optimal transportation plan and
call a $f\in Lip(1)$ satisfying the above supremum an optimal Kantorovich potential
transporting $\mu_1$ to $\mu_2$.

\begin{definition}\label{def:curvatures}
Let $G=(V, E, d, w)$ be a locally finite weighted graph associated with a weight distribution function $w$. Let $\mu_x^{\alpha}$ be the probability distribution function defined in equation (\ref{equ:defofdistri}) for any $0\leq \alpha \leq 1$.  For any $x, y\in V$, the $\alpha$-Ricci-curvature $\kappa_{\alpha}$ is defined  as
$$\kappa_{\alpha}(x, y)=1-\frac{W(\mu_x^{\alpha}, \mu_y^{\alpha})}{d(x, y)},  $$
where $W(\mu_x^{\alpha}, \mu_y^{\alpha})$ is the the Wasserstein distance transporting $\mu_x^{\alpha}$ to $\mu_y^{\alpha}$. 

The Ricci curvature $\kappa(x, y)$ is defined as
$$\kappa(x, y)=\lim\limits_{\alpha\to 1} \frac{k_{\alpha}(x, y)}{1-\alpha}. $$

The total curvature of $G$ is defined as
$$K(G)=\sum\limits_{xy\in E} \kappa(x, y). $$
\end{definition}

In the following  we state some basic properties of this generalized definition in the results Remark \ref{dnotintegerstilllinear} for $d$ restricted to a set of positive rational numbers and results in Theorem \ref{thm: lastlinear}.   These results are not logically necessary for this paper, the readers can skip this part and go directly to Theorem \ref{thm:curvatureviacoupling}. 

In the case of combinatorial graphs and $w=d=1$, $\kappa(x, y)$ is the Lin-Lu-Yau's curvature and $\kappa_{\alpha} (x, y)$ is Ollivier's curvature. 
In the $\alpha$-Ollivier-Ricci curvature,  for every edge $xy$ in $G$, the value $\alpha$ is called the {\it idleness},  and function $\alpha \to k_{\alpha}(x, y) $ is called the {\it Ollivier-Ricci idleness function}. The authors \cite{BCLMP} proved that the idleness function $\kappa_{\alpha}$ is a piece-wise linear function with at most three pieces. 
\begin{theorem}\cite{BCLMP}\label{thm:threelinearpiece}
Let $G=(V, E)$ be a locally finite graph. Let $x, y \in V$ such that $x\sim y$ and $d(x)\geq d(y)$. Then 
$\alpha \to k_{\alpha}(x, y) $ is a piece-wise linear function over $[0, 1]$ with at most 3 linear parts. 
Furthermore, $k_{\alpha}(x, y)$ is linear on $[0, \frac{1}{lcm(d(x), d(x))+1}]$ and is also linear on 
$[\frac{1}{\max(d(x), d(x))+1}, 1]$. Thus, if 
we have further condition $d(x)=d(y)$, then $k_{\alpha}(x, y)$ has at most two linear parts. 

\end{theorem}
One of two key ingredients of their proof in   \cite{BCLMP} is  the ``integer-valuedness" of optimal Kantorovich potentials which can be generalized to weighted graphs in our setting only if  the distance function $d$ is integer valued, the second one is the Complementary Slackness Theorem showing below which can be easily applied to weighted graphs. 

\begin{lemma}\cite{BCLMP}\label{fu-fv=duv}
Let $G = (V, E, d, w)$ be a locally finite weighted graph. Let $x, y \in  V$ with $x \sim y$. Let
$\alpha \in [0, 1]$. Let $A$ and $f$ be an optimal transport plan and an optimal Kantorovich potential transporting $\mu_x^{\alpha}$ to $\mu_y^{\alpha}$ respectively where $\mu_x^{\alpha}, \mu_y^{\alpha}, \kappa$ are defined in equation \ref{equ:defofdistri} and Def. \ref{def:curvatures}. Let $u, v \in  V$ with $A(u, v) \neq  0$. Then
$$ f(u) - f(v) = d(u, v).$$
\end{lemma}

In the following we assume the distance function $d$ is integer-valued. 
\begin{corollary}
Let $G = (V, E, d)$ be a locally finite weighted graph  with integer valued function $d$.  Let $f\in$ Lip(1), then $\floor f, \ceil f\in$ Lip(1). 
\end{corollary}
\begin{proof}
For each $v\in V$, set $\delta_v=f(v)-\floor f(v)$, then $\delta_v\in [0, 1)$ and for any $w\in V$, $\delta_v-\delta_w\in (-1, 1)$. We have 
\begin{equation*}
    |\floor f(v) -\floor f(w)| = |f(v)-\delta_v -f(w)+\delta_w| \leq d(v, w) + |\delta_v-\delta_w| < d(v, w) + 1. \end{equation*}
Since $d$ is a integer-valued function, then $|\floor f(v) -\floor f(w)| \leq  d(v, w)$. Thus $\floor f \in Lip(1)$. The proof that $\ceil f\in Lip(1)$ follows similarly.
\end{proof}

\begin{lemma}(Integer-Valuedness)\cite{BCLMP}%.\cite{BCLMP}  %does not work when distance is not integer valued function. 
Let $G = (V, E, d, w)$ be a locally finite weighted graph with integer valued function $d$. Let $x, y \in  V$ with $x \sim y$. Let
$\alpha \in [0, 1]$.  Then there exists $f \in $Lip(1) such that
$$W(\mu_x^{\alpha}, \mu_y^{\alpha})=\sup\limits_{f} \sum\limits_{u\in V} f(x)[\mu_x^{\alpha}(u)-\mu_y^{\alpha}(u)],$$
and $f(u)$ is an integer-valued function for all $u\in V$. 
\end{lemma}
\begin{proof}
The proof is omit, please refer to  Lemma 3.2 of \cite{BCLMP}. 
\end{proof}

\begin{corollary}\label{coro:piecewiselinear}
Let $G = (V, E, d, w)$ be a locally finite weighted graph  with integer valued function $d$. Let $x, y \in  V$ with $x \sim y$. Let $\alpha \in [0, 1]$,  $k_{\alpha}(x, y)$ is defined in Def. \ref{def:curvatures}, then $\alpha\to k_{\alpha}(x, y)$ is piece-wise linear over $[0, 1]$  with at most $2d(x, y)+1$ linear parts.
\end{corollary}

\begin{proof}
Let $f$ be an optimal Kantorovich potential with $f(y)=0$, then $f(x)$ could take at most $2d(x, y)+1$ integer values to satisfy $|f(x)-f(y)|\leq d(x, y)$. The proof is omit, refer to Theorem 3.3 in \cite{BCLMP}. 
\end{proof}

\begin{remark}\label{dnotintegerstilllinear}
Note if we assume $d$ is a set of positive rational numbers, we can re-scale the distance function $d$, for example by a multiple of $10$,  such that the $d$ is integer-valued, the curvature $\kappa(x, y)$ will not change by such a scaling. Thus $\alpha\to k_{\alpha}(x, y)$ is still a piece-wise linear function over $[0, 1]$.  
\end{remark}

\begin{theorem}\label{thm: lastlinear}
Let $G = (V, E, d, w)$ be a locally finite weighted graph associated by a weight distribution function $w$. Assume $d$ satisfies ``Treelike" explained in Section \ref{sec:sum}. For any  $x\sim  y$ with $D_x\geq D_y$, let
$\alpha \in (\frac{w_{xy}}{w_{xy}+D_x}, 1]$.  Let $f$ be an optimal Kantorovich potential transporting $\mu_x^{\alpha}$ to $\mu_y^{\alpha}$. Then 
$f(x)-f(y)=d(x, y)$. 
And $\alpha\to \kappa_{\alpha}(x, y)$ is linear in $[\frac{w_{xy}}{w_{xy}+D_x}, 1]$. 
\end{theorem}

\begin{proof}[Proof of Theorem \ref{thm: lastlinear}:]
Let $A$ be an optimal transport plan and $f$ be an optimal Kantorovich potential transporting $\mu_x^{\alpha}$ to $\mu_y^{\alpha}$. We only need to prove $f(x)-f(y)=d(x, y)$ and the rest is just similar as shown in  \cite{BCLMP}.  Since $\alpha> \frac{w_{xy}}{w_{xy}+D_x}$ and $D_x\geq D_y$, $\mu_x^{\alpha}(y)=\frac{(1-\alpha)w_{xy}}{D_x}< \alpha=\mu_y^{\alpha}(y)$, thus there exist vertex $z$ such that $A(z, y)>0$. If $z=x$ then $f(x)-f(y)=d(x, y)$ by Lemma \ref{fu-fv=duv}. Since $d$ satisfies ``Treelike", then there is no case $z\sim x$ and $z\sim y$, the only case left is consider $z\sim x$ and $z\neq x$. Then we have $f(z)-f(y)=d(z, y)$. Again by ``Treelike" $d(z, y)=d(z, x)+d(x, y)$. On the other hand, we have $f(z)-f(y)=f(z)-f(x) +f(x)-f(y) \leq d(z, x) +f(x)-f(y)$ which implies $f(x)-f(y)=d(x, y)$. 
The prove for the rest of theorem is similar as in Theorem 4.4 in \cite{BCLMP}. 
\end{proof}

If all distances are rational numbers in a weighted graph, with Remark \ref{dnotintegerstilllinear} and Theorem \ref{thm: lastlinear}, it is possible to compute the edge curvature by choosing $\alpha$ to be a value closer to $1$ in our settings.
\hfill

%\subsection{M\"{u}nch and Wojciechowski's Ricci curvature}
M\"{u}nch and Wojciechowski \cite{FlorentinRadoslaw} proposed a different generalized version of Lin-Lu-Yau Ricci curvature on weighted graph and also  expressed the curvature without a limit using graph Laplacian operator. What is different from our definition is all distances involved in their definition is 
the combinatorial distance,   i.e.  the distance  between any two vertices $x$ and $y$ is the minimum number of edges connecting $x$ and $y$. 
Now we briefly rephrase their probability distribution function and the result using our notations, note we  use $d'(x, y)$ to indicate the combinatorial distance, use $w(x, y)$ to represent the edge weight distribution.  

\begin{definition}\cite{FlorentinRadoslaw}
Let $G=(V, E, w)$ be a weighted graph with edge weight function $w$.
The probability distribution $\mu_x^{\alpha}$ be defined as:
\begin{align}\label{equ:defofdistriinotherpaper}
\mu_x^{\alpha}(z)=
\begin{cases}
 \alpha,  & \text{if $z=x$}, \\
 (1-\alpha)\frac{w(x, z)}{\sum\limits_{z\sim x} w(x, z)},  & \text{if $z\sim x$},\\
 0,& \text{otherwise}.
\end{cases}
\end{align}

For any  function $f: V\to \mathbb{R}$, the graph Laplacian $\Delta$ is defined by: 
$$\Delta f(x)=\frac{1}{\sum\limits_{z\sim x} w(x, z)}\sum\limits_{y\in V} w(x, y) (f(y)-f(x)). $$ 

And any two vertices $x, y$, let 
$$\nabla_{xy}f=\frac{f(x)-f(y)}{d'(x,y)}.  $$
\end{definition}

\begin{theorem}\label{thm:curvature_laplacian}\cite{FlorentinRadoslaw} (Curvature via the Laplacian) Let $G=(V, E, w)$ be a weighted graph with edge weight function $w$,  let $\mu_x^{\alpha}$ be the weight distribution function defined in Equation \ref{equ:defofdistriinotherpaper},  then for $x \neq y \in V(G)$, 
\begin{equation*}
    \kappa(x,y) = \inf_{\substack{f \in Lip(1)\\ \nabla_{yx}f = 1}} \nabla_{xy} \Delta f
\end{equation*}
\end{theorem}

Although the distance in \cite{FlorentinRadoslaw} is different, the proof still works in our setting when $d(x, y)$ is the weighted distance. 
\begin{corollary}\label{cor:curvature_laplacian}
Let $G=(V, E, d, w)$ be a weighted graph with edge weight function $d, w$. For any vertex $x\in V$,  let $\mu_x^{\alpha}$ be defined in expression \ref{equ:defofdistriinotherpaper}. Define $\nabla_{xy}f=\frac{f(x)-f(y)}{d(x,y)},$ where  $d(x, y)$ is the weighted distance. Let $\kappa$ be defined in Def. \ref{def:curvatures}, 
then for any $x \neq y \in V(G)$, 
\begin{equation*}
    \kappa(x,y) = \inf_{\substack{f \in Lip(1)\\ \nabla_{yx}f = 1}} \nabla_{xy} \Delta f.
     \end{equation*}
\end{corollary}

Motivated by Theorem \ref{thm:curvature_laplacian}, here we prove a dual theorem for a limit-free definition for our generalized version and thus for the Lin-Lu-Yau Ricci curvature. 
Let $\mu_x:=\mu_x^0$ be the probability distribution of random walk at $x$ with idleness equal to zero. For any two vertices $u$ and $v$, a {\em $\ast$-coupling} between $\mu_u$ and $\mu_v$ is a mapping $B: V\times V\to \mathbb{R}$ with finite support such that 
\begin{enumerate}
    \item $0<B(u,v)$, but all other values $B(x,y)\leq 0$.
    \item $\sum\limits_{x,y\in V} B(x,y)=0$.  
    %\item $\sum\limits_{y \in V} B(u, y)= 1$, $\sum\limits_{x \in V} B(x, v)=1$. True but not necessary.  
    \item $\sum\limits_{y \in V} B(x, y)=-\mu_u(x)$ for all $x$ except $u$. %(with item 3, 4, 5, item 2 is not needed)
    \item $\sum\limits_{x \in V} B(x, y)=-\mu_v(y)$ 
for all $y$ except $v$.  
\end{enumerate}
Because of items (2),(3), and (4), we get
$$B(u,v)=\sum_{(x,y)\in V\times V \setminus \{(u,v)\}}-B(x,y) \leq \sum_{x} \mu_u(x) + \sum_{y} \mu_v(y)\leq 2.$$

It is not hard to verify that the solutions exist for the  maximization of 
$\sum\limits_{x, y\in V} B(x, y)d(x, y)$,  considering the $\ast$-coupling $B(x, y)$ as variables in this linear programming problem,  as it is equivalent to the existence of solutions for the minimization of $W(\mu^\alpha_u, \mu^\alpha_v)$, see \cite{Ollivier1}.

\begin{theorem}(Curvature via Coupling function)\label{thm:curvatureviacoupling}
Let $G=(V, E, d, w)$ be a weighted graph with edge weight function $d, w$. $\kappa$ is defined in Def. \ref{def:curvatures}.  For any two vertex $u, v\in V$, we have
$$\kappa(u,v)=\frac{1}{d(u,v)}\sup\limits_{B} \sum\limits_{x, y\in V} B(x, y)d(x, y),$$
where the superemum is taken over all weak $\ast$-coupling $B$ between $\mu_u$ and $\mu_v$.
\end{theorem}

\begin{proof}
First we show 
\begin{equation}
\label{eq:kappa1}
 \kappa(u,v)\leq \frac{1}{d(u,v)}\sup\limits_{B} \sum\limits_{x, y\in V} B(x, y)d(x, y).   
\end{equation}

By Corollary \ref{coro:piecewiselinear}, for large enough $\alpha \in (0, 1)$, we have 
\begin{align*}
    \kappa(u,v) =\frac{\kappa_{\alpha}(u, v)}{1-\alpha} = \frac{1-\frac{W(\mu_u^{\alpha}, \mu_v^{\alpha})}{d(u, v)}}{1-\alpha}.
    \end{align*}
Let $A$ be the optimal coupling function transporting
$\mu^\alpha_u$ to $\mu^\alpha_v$.
Let $\mathbf{1}_{(u,v)}\colon V\times V \to {0,1}$ be the function taking value 1 at
$(u,v)$, and zero otherwise.
Let
$$B=\frac{1}{1-\alpha} ( \mathbf{1}_{(u,v)} - A).$$
It is straight forward to verify that $B$ is a $\ast$-coupling
between $\mu_u$ and $\mu_v$.

Thus, we have
\begin{align*}
 \frac{1}{d(u,v)} \sum\limits_{x, y\in V} B(x, y)d(x, y)
 &= \frac{1}{d(u,v)}\frac{1}{1-\alpha}
 \sum\limits_{x, y\in V} ( \mathbf{1}_{(u,v)}(x,y) - A(x,y))d(x, y)\\
 &= \frac{1}{d(u,v)}\frac{1}{1-\alpha}
 (d(u,v)-W(\mu^\alpha_u, \mu^\alpha_v))\\
 &=\frac{1}{1-\alpha}((1-\alpha)\kappa(u,v))\\
 &=\kappa(u,v).
\end{align*}
Thus \eqref{eq:kappa1} holds.

Now we prove the other direction
\begin{equation}
\label{eq:kappa2}
 \kappa(u,v)\geq \frac{1}{d(u,v)}\sup\limits_{B} \sum\limits_{x, y\in V} B(x, y)d(x, y).   
\end{equation}

Let $B'$ be the optimum $\ast$-coupling between $\mu_u$ and $\mu_v$. 
%Pick $\alpha\in (0,1)$ such that 
Choose a large enough $\alpha$ such that $\kappa(u,v) = \frac{\kappa_\alpha(u,v)}{1-\alpha}$ and
$(1-\alpha)B'(u,v)<1$.
Let $A=\mathbf{1}_{(u,v)}-(1-\alpha)B'$.
It is straightforward to verify that
$A$ is a coupling transporting 
$\mu^\alpha_u$ to $\mu^\alpha_v$.
Thus, we have

\begin{align*}
W(\mu^\alpha_v, \mu^\alpha_v)
&\leq \sum_{x,y} A(x,y)d(x,y)\\
&= \sum_{x,y} (\mathbf{1}_{(u,v)}-(1-\alpha)B')  d(x,y)\\
&= d(u,v)-(1-\alpha)\sum_{x,y}B'(x,y)d(x,y).
\end{align*}
Therefore, we have
\begin{align*}
% \kappa(u,v) &\geq \frac{\kappa^\alpha(u,v)}{1-\alpha}\\
\kappa(u,v) &= \frac{\kappa_\alpha(u,v)}{1-\alpha}\\
 &=\frac{1-\frac{W(\mu^\alpha_v, \mu^\alpha_v)}{d(u,v)}}{1-\alpha}\\
 &\geq \frac{1}{d(u,v)}\sum_{x,y}B'(x,y)d(x,y)\\
 &= \frac{1}{d(u,v)}\sup_B\sum_{x,y}B(x,y)d(x,y).
\end{align*}
The proof is complete. 
\end{proof}

\section{Sum of Ricci Curvatures}\label{sec:sum}
In this section, we study the sum of all edge curvatures
when the distance function $d$ satisfies the following property: 

\begin{property}[``Treelike"]
Let $G=(V, E, d)$ be a  weighted graph, we say $d$ satisfies ``Treelike" if  for any edge $(x, y)\in E$ and for any pair of vertices  $k \in N(x), l \in N(y)$, $d(k, l)=d(k, x)+d(x, y)+d(y, l)$.
\end{property}

A necessary condition for the existence of ``Treelike"  is girth of $G$ is at least 6.
Note when $G$ is a tree graph (finite or infinite), ``Treelike" clearly works for any distance function $d$. For non-tree graphs, one can easily verify that the girth of $G$ must be at least $6$(even if $d$ is not uniform). 
Clearly, there is no $3$-cycle supporting on each edge. Suppose there is a $4$-cycle,  we use $a, b, c, d$ to represent the edge length following one direction of the cycle, then we have $d= a+b+c$ and $b= a+d+c$ which imply that $0=2a+2c$, a contradiction.  Suppose there is a $5$-cycle with $a, b, c, d, e$ as the edge length following one direction of the cycle, then we have $d+e\geq a+b+c$ and $b+c\geq a+d+e$ which give us $0\geq 2a$, a contradiction. 
Suppose there is a $6$-cycle with $a, b, c, d, e, f$ as the edge length following one direction of the cycle, then we have $d+e+f \geq a+b+c$ and $a+b+c\geq d+e+ f$ which give us $c=f$, similarly, $a=d, b=e$.  There is no contradiction caused by the existence of cycles of length greater than $5$.  Thus any weighted graph $G$ satisfying ``Treelike" has $girth(G)\geq 6$.

Given a weight distribution function $w$, let $H(w)$ be the following quantity:
 \begin{align}
H(w) =\sum\limits_{uv\in E(G)} \Big( \frac{1}{D_u}\sum\limits_{x\in \Gamma(u)} \frac{w_{ux}d(u, x)}{d(u, v)} +   \frac{1}{D_v} \sum\limits_{y\in \Gamma(v)} \frac{w_{vy}d(v, y)}{d(u, v)}\Big).
\end{align}

For any weighted graph, we first prove
\begin{lemma}\label{lemma:general}
Let $G=(V, E, d, w)$ be a weighted graph associated by a weight distribution function $w$, then $K(G)\geq 2|V|  -H(w)$, with equality holds if and only if $d$ satisfies ``Treelike".
\end{lemma}

\begin{proof}

We fix an edge $uv\in E(G)$, recall $\Gamma (u)$ represents the set of neighbors of vertex $u$. Now we define a function $B: V\times V\to \mathbb{R}$. 
%Assume that $\frac{w_{uu}}{D_u}\leq \frac{w_{vv}}{D_v}$. % 
For any $x\in \Gamma(u)\setminus \{v\}$, let 
$B(x,y)=-\frac{w_{ux}}{D_u}$ if $y=v$ and 
 $0$ otherwise.
For any $y\in \Gamma(v)\setminus \{u\}$, let
$B(x,y)=-\frac{w_{vy}}{D_v}$ if $x=u$ and 
$0$ otherwise.
Let $B(v,v)=-\frac{w_{uv}}{D_u}$,
$B(u,u)=-\frac{w_{uv}}{D_v}$,
and $B(u,v)=2$.
%$B(u,v)=2-2\frac{w_{vv}}{D_v}-2\frac{w_{uu}}{D_u}$.
The rest of entries are set to $0$.
It is straightforward to verify the following results: 
\begin{center}
 $\sum\limits_{x,y\in V} B(x,y)=0$;  $\sum\limits_{y \in V} B(x, y)=-\mu_u(x)$ for all $x$ except $u$;  $\sum\limits_{x \in V} B(x, y)=-\mu_v(y)$ for all $y$ except $v$.
\end{center} 
Thus $B$ is $\ast$-coupling between $\mu_u$ and $\mu_v$. 
By Theorem \ref{thm:curvatureviacoupling}, we have 
\begin{align}\label{inequal:lowerofkuv}
\begin{split}
\kappa(u,v) &\geq \frac{1}{d(u,v)} \sum_{x,y\in V}B(x,y)d(x,y)\\
&= 2
- \sum_{x\in \Gamma(u)\setminus \{v\}} \frac{w_{ux}}{D_u} 
\frac{d(x,v)}{d(u,v)}
-\sum_{y\in \Gamma(v)\setminus \{u\}} \frac{w_{vy}}{D_v} 
\frac{d(u, y)}{d(u,v)}\\ 
&\geq 2
- \sum_{x\in \Gamma(u)\setminus \{v\}} \frac{w_{ux}}{D_u} 
\frac{d(x,u)+d(u,v)}{d(u,v)}
-\sum_{y\in \Gamma(v)\setminus \{u\}} \frac{w_{vy}}{D_v} 
\frac{d(y,v)+d(u,v)}{d(u,v)}\\ 
&= 2
- \sum_{x\in \Gamma(u)\setminus \{v\}} \frac{w_{ux}}{D_u} 
-\sum_{y\in \Gamma(v)\setminus \{u\}} \frac{w_{vy}}{D_v} 
-
\sum_{x\in \Gamma(u)\setminus \{v\}} \frac{w_{ux}}{D_u} 
\frac{d(x,u)}{d(u,v)}
-\sum_{y\in \Gamma(v)\setminus \{u\}} \frac{w_{vy}}{D_v} 
\frac{d(y,v)}{d(u,v)}\\ 
&= \frac{w_{uv}}{D_u} + \frac{w_{uv}}{D_v} 
- \sum_{x\in \Gamma(u)\setminus \{v\}} \frac{w_{ux}}{D_u} 
\frac{d(x,u)}{d(u,v)}
-\sum_{y\in \Gamma(v)\setminus \{u\}} \frac{w_{vy}}{D_v} 
\frac{d(y,v)}{d(u,v)}\\
&= \frac{2w_{uv}}{D_u} + \frac{2w_{uv}}{D_v} 
- \sum_{x\in \Gamma(u)} \frac{w_{ux}}{D_u} 
\frac{d(x,u)}{d(u,v)}
-\sum_{y\in \Gamma(v)} \frac{w_{vy}}{D_v} 
\frac{d(y,v)}{d(u,v)}.
\end{split}
\end{align}

Therefore,  we have
\begin{align}\label{inequ:totalcurvature}
\begin{split}
\sum_{uv\in E(G)}\kappa(u,v)&\geq  \sum_{uv\in E(G)} \Big\{\frac{2w_{uv}}{D_u} + \frac{2w_{uv}}{D_v} -\sum_{x\in N(u)} \frac{w_{ux}}{D_u} 
\frac{d(x,u)}{d(u,v)}
-\sum_{y\in N(v)} \frac{w_{vy}}{D_v} 
\frac{d(y,v)}{d(u,v)}
\Big\}\\
&= \sum\limits_{u\in V(G)} \sum\limits_{v\sim u} \frac{2w_{uv}}{D_u}- \sum\limits_{u\sim v}\Big( \frac{1}{D_u}  \sum\limits_{x\in \Gamma(u)} \frac{w_{ux}d(u,x)}{d(u, v)}+   \frac{1}{D_v} \sum\limits_{y\in \Gamma(v)} \frac{w_{vy}d(v, y)}{d(y, v)} \Big)\\
&= 2|V|  -\sum\limits_{u\sim v}\Big( \frac{1}{D_u}  \sum\limits_{x\in \Gamma(u)} \frac{w_{ux}d(u,x)}{d(u, v)}+   \frac{1}{D_v} \sum\limits_{y\in \Gamma(v)} \frac{w_{vy}d(v, y)}{d(y, v)} \Big). 
    \end{split}
\end{align}

The proof for $K(G)\geq 2|V|  -H(w)$ is complete. Next, we characterize the equality condition of this inequality. 
For equation in (\ref{inequ:totalcurvature}) holds , both equations in  (\ref{inequal:lowerofkuv}) must hold. For the second equation in  (\ref{inequal:lowerofkuv}), we must have for any edge $uv\in E(G)$, $d(x, v)= d(x, u)+d(u, v)$ for all $x\in \Gamma(u)\setminus \{v\}$ and  $d(u, y)=d(u, v)+d(v, y)$ for all $y\in \Gamma(v)\setminus \{u\}$. We further verify that $G$ satisfies  ``Treelike". 

Suppose there exists an edge $uv\in E(G)$ such that  ``Treelike" fails, and let  $x\in \Gamma(u)\setminus \{v\}, y\in \Gamma(v)\setminus \{u\}$ be the two vertices so that $d(x, y)< d(x, u)+d(u, v)+d(v, y)$. 
We will show the first inequality in (\ref{inequal:lowerofkuv}) would be strict by defining a new $\ast$-coupling function $B'$, which violates the equation $K(G)=2|V|-H(w)$. 

WLOG assume that $\frac{w_{vu}}{D_v}\geq \frac{w_{uv}}{D_u}$. Let $B'(x, y)=-\frac{w_{ux}}{D_u}$, $B'(u, y)=-\frac{w_{vy}}{D_v}+ \frac{w_{ux}}{D_u}$,  $B'(u, v)=2-\frac{w_{ux}}{D_u}$ and $B'(x, v)=0$,  for other entries $B'=B$.  It is easy to verify that $B'$ is a $\ast$-coupling. 
We have 
\begin{align*}
    & \sum_{x,y\in V}B'(x,y)d(x,y) - \sum_{x,y\in V}B(x,y)d(x,y)\\
    &= (B'(x,y)-B(x, y))d(x,y)+ (B'(u, y)-B(u, y))d(u, y) +(B'(u, v)-B(u, v))d(u, v) +(B'(x, v)-B(x, v))d(x, v)\\
    &=(-\frac{w_{ux}}{D_u}-0)d(x,y)+ (-\frac{w_{vy}}{D_v}+ \frac{w_{ux}}{D_u}--\frac{w_{vy}}{D_v})d(u, y) +(2-\frac{w_{ux}}{D_u}-2)d(u, v) +(0--\frac{w_{ux}}{D_u})d(x, v)\\
    &=-\frac{w_{ux}}{D_u}d(x,y)+ \frac{w_{ux}}{D_u}d(u, y) -\frac{w_{ux}}{D_u}d(u, v) +\frac{w_{ux}}{D_u}d(x, v)\\
    &= \frac{w_{ux}}{D_u}\Big(-d(x,y) +d(u, y)-d(u, v) +d(x, v)\Big)\\
    &>\frac{w_{ux}}{D_u}\Big(-d(x, u)-d(u, v)-d(v, y)+d(u, y)-d(u, v) +d(x, v)\Big)\\
    &=\frac{w_{ux}}{D_u}\Big(-d(x, u)-2d(u, v)-d(v, y)+d(u, v)+d(v, y) +d(x, u)+d(u, v)\Big)\\
    %&=\frac{w_{ux}}{D_u}\Big(-d(x, u)+d(x, u)\Big)\\
    &=0.
    % &= -\frac{w_{ux}}{D_u}d(u, v) -\frac{w_{ux}}{D_u} d(x, y) + \frac{w_{ux}}{D_u}d(x, v) -(-\frac{w_{vy}}{D_v}+ \frac{w_{ux}}{D_u})d(u, y)+\frac{w_{vy}}{D_v}d(u, y)\\
   % &=\frac{w_{ux}}{D_u} \Big(d(x, v)-d(u, v)-d(x, y)-d(u, y)\Big) + \frac{w_{vy}}{D_v}\Big(d(u, y)+d(u, y)\Big)\\
    %& =\frac{w_{ux}}{D_u} \Big((d(x, u)+d(u, v))-d(u, v)-d(x, y)-(d(u, v)+d(v, y))\Big) + 2\frac{w_{vy}}{D_v}d(u, y)\\
   % &> \frac{w_{ux}}{D_u} \Big(d(x, u)-(d(x, u)+d(u, v)+d(v, y))-(d(u, v)+d(v, y))\Big) + 2\frac{w_{vy}}{D_v}d(u, y)\\
  % & = \frac{w_{ux}}{D_u} \Big(-d(u, v)-d(v, y)-d(u, v)-d(v, y)\Big) + 2\frac{w_{vy}}{D_v}d(u, y)\\
    %& = 2\frac{w_{ux}}{D_u} \Big(-d(u, v)-d(v, y)\Big) + 2\frac{w_{vy}}{D_v}d(u, y)\\
   % & = 2\frac{w_{ux}}{D_u} \Big(-d(u,y)\Big) + 2\frac{w_{vy}}{D_v}d(u, y)\\
  %  & = 2 d(u,y)\Big(\frac{w_{vy}}{D_v}- \frac{w_{ux}}{D_u} \Big)\\
   %  &\geq 0. 
\end{align*}

Then 
\begin{align*}
    \kappa(u, v)& \geq \frac{1}{d(u,v)} \sum_{x,y\in V}B'(x,y)d(x,y) \\
    & > \frac{2w_{uv}}{D_u} + \frac{2w_{uv}}{D_v} 
- \sum_{x\in \Gamma(u)} \frac{w_{ux}}{D_u} 
\frac{d(x,u)}{d(u,v)}
-\sum_{y\in \Gamma(v)} \frac{w_{vy}}{D_v} 
\frac{d(y,v)}{d(u,v)}, 
\end{align*}
a contradiction to the assumption that equation in (\ref{inequ:totalcurvature}) holds.

Thus ``Treelike" is a necessary condition for $K(G)=2|V|  -H(w)$.  Next, we prove the  lower bound in (\ref{inequal:lowerofkuv}) is tight under ``Treelike". Define function $f: V\to \mathbb{R}$ such that 
$f(u)=0$, $f(x)=-d(x, u)$ for $x\in \Gamma(u)\setminus\{v\}, f(v)=d(u, v)$  and $f(y)=d(u, v)+d(v,y)$ for $y\in \Gamma(v)\setminus\{u\}$. It is straightforward to verify that $f$ is a Lip(1) function.  In addition,  $\nabla_{vu}f = 1$, by Corollary \ref{cor:curvature_laplacian},  we have

\begin{align*}\label{inequ:fisyipoptimaypropertvI}
    \kappa(u, v)& \leq  \frac{\Delta f(u)- \Delta f(v)}{d(u, v)} \\
    & =  \frac{1}{d(u, v)} \Big( \sum\limits_{x\in \Gamma(u)\setminus\{v\}} \frac{w_{ux}}{D_u} (-d(u, x)) + \frac{w_{uv}}{D_u} (d(u, v))-\sum\limits_{y\in \Gamma(v)\setminus\{u\}} \frac{w_{vy}}{D_v} (d(u, v)+d(v,y)-d(u,v))\\
    & \quad\quad\quad\quad\quad\quad- \frac{w_{vu}}{D_v} (-d(u, v))\Big)\\
    & =\frac{w_{uv}}{D_u}+ \frac{w_{vu}}{D_v}-\sum\limits_{u\in \Gamma(u)\setminus\{v\}} \frac{w_{ux}}{D_u} \frac{d(u,x)}{d(u, v)}  -\sum\limits_{y\in \Gamma(v)\setminus\{u\}} \frac{w_{vy}}{D_v}  \frac{d(v, y)}{d(u, v)}\\
    &= \frac{2w_{uv}}{D_u} + \frac{2w_{uv}}{D_v} 
- \sum_{x\in \Gamma(u)} \frac{w_{ux}}{D_u} 
\frac{d(x,u)}{d(u,v)}
-\sum_{y\in \Gamma(v)} \frac{w_{vy}}{D_v} 
\frac{d(y,v)}{d(u,v)}. 
\end{align*}
Therefore, ``Treelike" is a sufficient condition for $K(G)=2|V|-H(w)$. The proof is complete. 
\end{proof}

\subsection{The minimum of K(G) under a certain weight distribution}
In this section,  we study a case of the weight distribution function $w$ where the total curvature achieves the minimum at the uniform distance function.

\begin{theorem}\label{maintheorem}
Let $G=(V, E, d, w)$ be a weighted graph associated by the weight distribution function $w$, where  $w_{e}=\frac{F(d(e))}{d(e)}$ for each edge $e\in E$, and $F(x)$ is  a non-increasing function on $\mathbb{R}^{+}$. Then the total curvature $K(G)\geq 2|V|-2|E|$ with equality holds if and only one of the following two conditions is true:
\begin{enumerate}
\item  the weight distance function $d$ is uniform and  $girth(G)\geq 6$.
\item   $d$ satisfies ``Treelike" and $F$ is a constant function. 
\end{enumerate}
\end{theorem}

To prove Theorem \ref{maintheorem}, we need the following lemma. 
\begin{lemma}\label{lemma:inequalityaboutw_xy}
Let $G=(V, E, d)$ be a weighted graph associated by a weight distribution function $w$  with $w_{e}=\frac{F(d(e))}{d(e)}$ for each edge $e\in E$. 
If  $F(x)$ is an non-increasing function on $\mathbb{R}^{+}$, then the following is true.
\begin{align}
  \frac{w_{e}d(e)}{d(f)} + \frac{w_{f}d(f)}{d(e)} \leq w_{e}+ w_{f}.
\end{align}
\end{lemma}

\begin{proof}
Taking the subtraction of two sides, it is sufficient to show the following inequality
\begin{align*}
&w_{e}d(e)^2 + w_{f}d(f)^2 - w_{e}d(f)d(e)- w_{f}d(f)d(e)\\
&= ( F(d(e))-F(d(f)) ) d(e) + (F(d(f))- F(d(e)))d(f)\\
&=( F(d(e))-F(d(f)) ) (d(e)-d(f))\\
&\leq 0,
\end{align*}
which is true as  $F(x)$ is an non-increasing function.
\end{proof}

\begin{proof}[Proof of Theorem \ref{maintheorem}:]
By Lemma \ref{lemma:inequalityaboutw_xy}, we have 
\begin{align*}
    \begin{split}
        &\sum\limits_{uv\in E(G)} \Big\{  \frac{1}{D_u} \sum\limits_{x, v\in \Gamma(u)} \Big(\frac{w_{ux}d(u, x)}{d(u, v)}+ \frac{w_{uv}d(u, v)}{d(u, x)} \Big) + \frac{1}{D_v}\sum\limits_{u, y\in \Gamma(v)} \Big( \frac{w_{vy}d(v, y)}{d(u, v)} + \frac{w_{vu}d(u, v)}{d(v, y)}\Big)\Big\}\\
        & \leq \sum\limits_{uv\in E(G)} \Big\{\frac{1}{D_u} \sum\limits_{x, v\in \Gamma(u)} \Big(w_{ux} + w_{uv}\Big) + \frac{1}{D_v} \sum\limits_{u, y\in \Gamma(v)} \Big(w_{vu}+ w_{vy}\Big) \Big\}\\
        & =\sum\limits_{uv\in E(G)}  (2+2)\\
        &=4|E|.  
     \end{split}
\end{align*}

Following the Lemma \ref{lemma:general}, we have 
\begin{align}\label{inequ:totalminimalp}
\begin{split}
K(G) & \geq \sum\limits_{uv\in E(G)} \Big\{\frac{2w_{uv}}{D_u} +  \frac{2w_{vu}}{D_v} 
-\sum\limits_{x\in \Gamma(u)} \frac{w_{ux}d(u, x)}{D_u d(u, v)} -  \sum\limits_{y\in \Gamma(v)} \frac{w_{vy}d(v, y)}{D_y d(u, v)} \Big\}\\
&= 2 |V|-\frac{1}{2} \sum\limits_{uv\in E(G)} \Big\{  \frac{1}{D_u} \sum\limits_{x, v\in \Gamma(u)} \Big(\frac{w_{ux}d(u, x)}{d(u, v)}+ \frac{w_{uv}d(u, v)}{d(u, x)} \Big)\\
& \quad\quad\quad\quad+  \frac{1}{D_v}\sum\limits_{u, y\in \Gamma(v)} \Big( \frac{w_{vy}d(v, y)}{d(u, v)} + \frac{w_{vu}d(u, v)}{d(v, y)}\Big)\Big\}\\
& \geq  2|V|-2|E|.
\end{split}
\end{align}

Recall the first equation in inequality (\ref{inequ:totalminimalp}) holds if and only if $G$ satisfies ``Treelike" by Lemma \ref{lemma:general}; for the second equation holds, there are two cases:  the distance function $d$ is uniform over $E(G)$(then $F$ is a constant function automatically); if $d$ is not uniform, let $x$ be the vertex such that there exist two edges $e, f$ incident to $x$ with $d(e)>d(f)$, then we must have $F(d(e))=F(d(f))$. 
WLOG, let $d(f)=\min \{d(f'): f'\in E(G)\}$, then $F$ is a constant function over $E(G)$ with $F=F(d(f))$.

\end{proof}

\subsection{The maximum of $K(G)$ under a certain weight distribution}
The inequality of $K(G)$ in Theorem \ref{maintheorem} can be reversed if $F(\bullet)$ is an increasing function and $d$ satisfies ``Treelike''. Similar to Lemma \ref{lemma:inequalityaboutw_xy}, we have

\begin{lemma}\label{lemma:inequalityaboutw_xywhenFnondecreasing}
Let $G=(V, E, d)$ be a weighted graph associated by a weight distribution function $w$  with $w_{e}=\frac{F(d(e))}{d(e)}$ for each edge $e\in E$. 
If  $F(x)$ is an increasing function on $\mathbb{R}^{+}$, then the following is true. The equlity holds if and only if $d(e)=d(f)$.
\begin{align}
  \frac{w_{e}d(e)}{d(f)} + \frac{w_{f}d(f)}{d(e)} \geq w_{e}+ w_{f}.
\end{align}
\end{lemma}

\begin{theorem}\label{theom:w=dmaximaltheorem}
Let $G=(V, E, d, w)$ be a weighted graph such that the distance function $d$ satisfies ``Treelike" and  $F(\bullet)$ is an increasing function, then the total curvature $K(G)$ is at most $2|V|-2|E|$ with equality holds if and only if $d$ is a constant function.
\end{theorem}

\begin{proof}
By Lemma \ref{lemma:inequalityaboutw_xywhenFnondecreasing}, we have 
\begin{align*}
    \begin{split}
        &\sum\limits_{uv\in E(G)} \Big\{  \frac{1}{D_u} \sum\limits_{x, v\sim u} \Big(\frac{w_{ux}d(u, x)}{d(u, v)}+ \frac{w_{uv}d(u, v)}{d(u, x)} \Big) + \frac{1}{D_v}\sum\limits_{u, y\sim v} \Big( \frac{w_{vy}d(v, y)}{d(u, v)} + \frac{w_{vu}d(u, v)}{d(v, y)}\Big)\Big\}\\
        & \geq \sum\limits_{uv\in E(G)} \Big\{\frac{1}{D_u} \sum\limits_{x, v\sim u} \Big(w_{ux} + w_{uv}\Big) + \frac{1}{D_v} \sum\limits_{u, y\sim v} \Big(w_{vu}+ w_{vy}\Big) \Big\}\\
        & =\sum\limits_{uv\in E(G)}  (2+2)\\
        &=4|E|.  
     \end{split}
\end{align*}
Clearly, the equality holds if and only if $d$ is a constant function. 

As $G$ is ``Treelike", following the Lemma \ref{lemma:general}, we have 
\begin{align}\label{inequ:totalminimal}
\begin{split}
K(G) &=2|V|-H(w)\\
&= 2 \sum\limits_{u\in V(G)} \sum\limits_{v\sim u} \frac{w_{uv}}{D_u}-\frac{1}{2} \sum\limits_{uv\in E(G)} \Big\{  \frac{1}{D_u} \sum\limits_{x, v\sim u} \Big(\frac{w_{ux}d(u, x)}{d(u, v)}+ \frac{w_{uv}d(u, v)}{d(u, x)} \Big)\\
& \quad\quad\quad\quad+  \frac{1}{D_v}\sum\limits_{u, y\sim v} \Big( \frac{w_{vy}d(v, y)}{d(u, v)} + \frac{w_{vu}d(u, v)}{d(v, y)}\Big)\Big\}\\
& \leq  2|V|-2|E|.
\end{split}
\end{align}
\end{proof}

\begin{corollary}\label{coro:Gauss}
Let $G$ be a finite graph with girth at least $6$ with uniform edge weights. Then Gauss-Bonnet theorem holds for $G$. i.e. $K(G)=\chi(G)$, where $\chi(G)=2-2g(G)$ is the Euler characteristic of $G$, and where $g(G)=|E|-|V|+1$ is the graph genus of $G$. 
\end{corollary}
\begin{proof}
$\chi(G)=2-2g(G)=2-2(|E|-|V|+1)=2|V|-2|E|=K(G)$, where the last equality follows from Theorem \ref{theom:w=dmaximaltheorem}.
\end{proof}

\begin{remark}
Note that the graph curvature here is defined geometrically, via optimal transport, in contrast to previous combinatorial definitions of graph curvature used in versions of the graph Gauss-Bonnet theorem \cite{Peichen}. Intuitively, speaking of the Ricci curvature, the above corollary says that an unweighted graph with girth at least $6$ behaves like a closed surface.
\end{remark}

\bibliographystyle{plain}
\bibliography{totalcurvature}

\end{document}